\documentclass[onecolumn,12pt]{IEEEtran}
\usepackage{graphicx}
\usepackage{caption}
\usepackage{subcaption}
\usepackage{savesym}
\savesymbol{AND}
\usepackage[numbers]{natbib}
\usepackage{amssymb}
\usepackage{amsmath}
\usepackage{dsfont}
\usepackage[usenames,dvipsnames]{xcolor}
\usepackage{graphicx}
\usepackage{graphics}
\usepackage{algorithm}
\usepackage{algorithmic}
\usepackage{todonotes}
\usepackage{booktabs}
\usepackage{caption}
\usepackage{mathtools}
\usepackage{comment}
\usepackage{hyperref}

\newcommand {\be}{\begin{equation}}
\newcommand {\ee}{\end{equation}}

\newcommand {\beW}{\begin{equation*}}
\newcommand {\eeW}{\end{equation*}}

\newcommand {\bsp}{\begin{split}}
\newcommand {\esp}{\end{split}}

\newcommand {\bea}{\begin{eqnarray}}
\newcommand {\eea}{\end{eqnarray}}

\newcommand {\beaW}{\begin{eqnarray*}}
\newcommand {\eeaW}{\end{eqnarray*}}

\newcommand {\bM}{\begin{bmatrix}}
\newcommand {\eM}{\end{bmatrix}}

\newcommand {\bi}{\begin{itemize}}
\newcommand {\ei}{\end{itemize}}

\newcommand {\ben}{\begin{enumerate}}
\newcommand {\een}{\end{enumerate}}

\newcommand {\bal}{\begin{align}}
\newcommand {\eal}{\end{align}}

\newcommand {\balW}{\begin{align*}}
\newcommand {\ealW}{\end{align*}}

\newcommand {\bcm}{\begin{columns}}
\newcommand {\ecm}{\end{columns}}
\newcommand {\bc}{\begin{column}}
\newcommand {\ec}{\end{column}}

\newcommand {\eps}{\varepsilon}

\newcommand{\m}{\mathcal}

\newtheorem{thm}{Theorem}
\newtheorem{prop}{Proposition}
\newtheorem{pf}{Proof}
\newtheorem{lemma}{Lemma}

\newtheorem{assumption}{Assumption}

\DeclareMathOperator{\dx}{dx}

\DeclareMathOperator{\Xs}{\mathcal{X}}
\DeclareMathOperator{\F}{\mathcal{F}}

\DeclareMathOperator{\XUs}{\mathcal{X}\times\mathcal{U}}

\DeclareMathOperator{\Us}{\mathcal{U}}
\DeclareMathOperator{\Prb}{\mathbb{P}}

\DeclareMathOperator{\erf}{erf}

\DeclareMathOperator{\st}{subject\ to}

\DeclarePairedDelimiter\parens{\lparen}{\rparen}
\DeclarePairedDelimiter\bracks{\{}{\}}
\DeclarePairedDelimiter\sqbracks{[}{]}

\begin{document}

\title{The Linear Programming Approach to  Reach-Avoid Problems for Markov Decision Processes}

\author{Nikolaos Kariotoglou, Maryam Kamgarpour, Tyler Summers, and John Lygeros}
\maketitle

\begin{abstract}
One of the most fundamental problems  in Markov decision processes is analysis and control synthesis for safety and reachability specifications. We consider the stochastic reach-avoid problem, in which the objective is to synthesize a control policy to maximize the probability of reaching a target set at a given time, while staying in a safe set at all prior times.  We characterize  the solution to this problem through an infinite dimensional linear program. We then develop a tractable approximation to the infinite dimensional linear program through finite dimensional approximations of the decision space and constraints. For a large class of Markov decision processes modeled by Gaussian mixtures kernels we show that through a proper selection of the finite dimensional space, one can further reduce the computational complexity of the resulting linear program. We validate the proposed method and analyze its potential with a series of numerical case studies.   %
\end{abstract}

\section{Introduction}
\label{sec:intro}
A wide range of controlled dynamical systems can be modeled using the framework of Markov decision processes (MDPs) \cite{feinberg2002handbook,puterman1994markov}. Depending on the problem at hand, several objectives can be formulated for an MDP including maximization of a reward function or satisfaction of a specification defined by a formal language. Safety and reachability are two of the most fundamental specifications for a dynamical system.  In a reach-avoid problem for an MDP, the objective is to maximize the probability of reaching a target set within a given time horizon while staying in a safe set \cite{abate2008probabilistic}. This objective is stage-wise \emph{sum-multiplicative}, in contrast with the stage-wise additive cost functions typically used in MDPs. This addresses a recognized limitation of additive cost functions: many tasks are not easily encoded by an additive cost function and are more naturally posed in terms of reaching and avoiding certain sets. This difficulty is evidenced by the problem of inverse reinforcement learning \cite{ng2000algorithms,abbeel2004apprenticeship,abbeel2011inverse}, a well-known problem in artificial intelligence where the objective is to \emph{learn} a reward or cost function being optimized based on observed behavior of an agent/controller,  in tasks where it is not entirely obvious what should be optimized. The stochastic reach-avoid framework has been applied to several problems including aircraft conflict detection under stochastic wind \cite{watkins2003stochastic,ding2013stochastic}, feedback control of camera networks in the presence of an uncertain evader \cite{kariotoglou2014camaeras} and optimal feedback policies for building evacuation under a randomly evolving hazards \cite{wood2013stochastic}. 

The reach-avoid problem considered in this paper is closely related to the stochastic shortest path problem \cite{bertsekas1991analysis}. In contrast to stochastic shortest path, however, there is no cost function for transitioning from one state (often treated as a graph node) to another. As pointed out in \cite{kolobov2011heuristic,kolobov12mausam}, this difference makes the dynamic programming algorithm developed for stochastic shortest path to fail for the reach-avoid problem and certain problem instances with more general reward structures, than the mostly studied additive reward functions. Hence, the authors in \cite{kolobov2011heuristic,kolobov12mausam} propose the so-called generalized stochastic shortest path framework that can address a wide range of stage cost structures. It can, in particular, address the problem of maximizing the probability to reach a goal set. Our problem, though has very similar objective, is not formulated in the category of MDPs considered in \cite{kolobov2011heuristic,kolobov12mausam} due to (a) the continuous state-space and (b) a fixed finite horizon to reach the target set. Both of these considerations are motivated by engineering applications in which one has limited time to achieve an objective and furthermore, modeling the continuous dynamical system with finite state space would be prohibitive due to explosion of number of states.

The dynamic programming (DP) principle characterizes the solution to the stochastic reach-avoid problem  with continuous state and action spaces \cite{prandini2006stochastic}. However, it is intractable to find the reach-avoid value function through the DP equations. One  can  approximate the DP equations on a finite grid defined over the MDP state and action spaces. Gridding techniques are theoretically attractive since they can provide explicit error bounds for the approximation of the value function under general Lipschitz continuity assumptions \cite{abate2007computational,kushner2001numerical}. In practice, the complexity of gridding based techniques suffer from the infamous Curse of Dimensionality. That is, the sum of state and control space dimensions that can be addressed is limited by the cardinality of the state-control pairs that need to be considered to fairly approximate reach-avoid probabilities. Typically, the required cardinality to keep approximations meaningful scales exponentially with  dimensions of state and action spaces. An important problem is therefore to explore approximation techniques that scale better.

Several researchers have developed approximate dynamic programming (ADP) techniques for various classes of stochastic control problems \cite{powell2007approximate,bertsekas2012dynamic}. Most of the existing work has focused on problems where the state and control spaces are finite but too large to directly solve DP recursions. Our work is motivated by the technique discussed in \cite{de2003linear} where the authors develop an ADP method for optimal control of an MDP with finite state and action spaces and an infinite horizon discounted additive stage cost. In this approach, the value function of the stochastic control problem is approximated as a weighted sum of basis functions, where the weights are the solution to a linear program (LP) \cite{de2004constraint}. The number of constraints in the LP is equal to the cardinality of state and action spaces. Hence, computation becomes challenging for large MDPs. To handle this, a constraint sampling approach with probabilistic bounds has been proposed in \cite{de2004constraint}. 

For optimal control of MDPs with continuous state and action spaces and an additive stage cost, an infinite dimensional linear program has been developed to characterize the value function \cite{hernandez1996discrete}. Here, the decision variable is the value function defined over the uncountable state space, hence, it is  infinite dimensional. Furthermore,   the number of constraints is uncountably infinite since there is one constraint corresponding to each state, action pair. The authors in \cite{hauskrecht2003linear, kveton2006solving} consider a similar setup extending to mixed continuous and discrete state variables. They also propose approximating the value function (the infinite dimensional decision variable) as a weighted sum of basis functions and devise an efficient approach to solving the resulting large-scale LP by considering dynamical systems that are modeled by or can be fairly approximated using the so-called ``factored'' MDPs. In contrast to this line of work, we address a finite-horizon reach-avoid problem over continuous state and action spaces and no discount factor. Furthermore, we make no a-priori assumption on whether the system dynamics can be factored.

The LP approach to stochastic reachability problem for MDPs over continuous state and action spaces and an infinite horizon was first proposed in \cite{kamgarpour2013HSCC}. An infinite dimensional linear program was formulated whose solution, in theory, would characterize the maximum reachability probability over the continuous state space. However, no computational approach to solving this problem was provided. In general, LP approaches to ADP are desirable since several commercially available software packages can handle LP problems with large numbers of decision variables and constraints. Motivated by this observation and leveraging advances in the past works of \cite{powell2007approximate,bertsekas2012dynamic, kamgarpour2013HSCC} we develop a computational framework to approximate the optimal value function and policy of a stochastic reach-avoid problem over continuous state and action spaces. 

Our contributions are as follows: First, we derive an infinite dimensional LP formulated over the space of Borel measurable functions  and prove its equivalence to the standard DP-based solution approach for the stochastic reach-avoid problem, under assumptions of the continuity of the MDP transition kernel and compactness of the action space. Second, we prove that through restricting the infinite dimensional decision space to a finite dimensional subspace spanned by a collection of basis functions (semi-infinite or robust LP), we obtain an upper bound on the stochastic reach-avoid value function. Third, we use randomized optimization to obtain a tractable finite dimensional LP with probabilistic feasibility guarantees. The final contribution of our paper is the focus on numerical validation of the LP approach to stochastic reach-avoid problems. As such, we propose a class of basis functions for reach-avoid problems for MDPs with Gaussian mixture kernels. Basis functions in this class have been successfully used in similar function approximation schemes such as \cite{kveton2006learning} due to their analytic properties. We then develop several benchmark problems to test the scalability and accuracy of the method. 

A preliminary version of our approach appeared as a brief conference paper \cite{kariotoglou2013approximate}. Compared to \cite{kariotoglou2013approximate}, we have extended and refined all theoretical statements. The Lemmas, Propositions and Theorems presented here were either missing from \cite{kariotoglou2013approximate} or not proven.  Furthermore, we have provided novel numerical studies to  illustrate the accuracy of the approach and its applicability to relatively large-scale problems compared to reach-avoid problems handled in the literature. Given that there are no competing approaches for the problem at hand to handle the considered large state-input dimensions, we have compared the results to well-studied heuristics, specifically tuned to approximate the solution to simple stochastic reach-avoid problems.



The rest of the paper is organized as follows. In Section \ref{sec:stochreach_DP} we introduce the stochastic reach-avoid problem for MDPs and formulate an infinite dimensional LP that characterizes its solution. In Section \ref{sec:lpcampi} we derive an approach to approximate the solution to the infinite LP through restricting the decision space to a finite dimensional subspace using basis functions and reducing the infinite constraints to finite constraints through randomized sampling. Section \ref{sec:implement} proposes Gaussian radial basis functions to analytically compute  operations arising in the LP for MDPs with Gaussian mixture kernels. In Section \ref{sec:evaluate} we validate the accuracy and scalability of the solution approach with three case studies. 

\section{Stochastic reach-avoid problem}
\label{sec:stochreach_DP}
We consider a discrete-time controlled stochastic process $x_{t+1}\sim Q(dx|x_t,u_t)$, $(x_t,u_t)\in \Xs \times\Us $. Here,  $Q:\mathcal{B}(\Xs)\times\Xs\times\Us\rightarrow [0,1]$ is a transition kernel and  $\mathcal{B}(\mathcal{X})$ denotes the Borel $\sigma$-algebra of $\mathcal{X}$. Given a state control pair $(x_t,u_t)\in \mathcal{X}\times\mathcal{U}$, $Q(A|x_t,u_t)$ measures the probability of $x_{t+1}$ belonging to the set $A\in\mathcal{B}(\Xs)$. The transition kernel $Q$ is a Borel-measurable stochastic kernel, that is, $Q(A|\cdot)$ is a Borel-measurable function on $\m{X} \times \m{U}$ for each $ A \in \m{B}(\m{X})$ and $Q(\cdot|x,u)$ is a probability measure on $\m{X}$ for each $(x,u)$. For the rest of the paper all measurability conditions refer to Borel measurability. We allow the state space $\mathcal{X}$ to be any subset of $\mathbb{R}^n$ and assume that the control space $\mathcal{U}\subseteq \mathbb{R}^m$  is compact. 

We consider a safe set $K^{\prime}\in\mathcal{B}(\mathcal{X})$ and a target set $K\subseteq K^{\prime}$. We define an admissible $T$-step control policy to be a sequence of measurable functions $\mu = \{\mu_0,\dots,\mu_{T-1}\}$ where $\mu_i:\mathcal{X}\rightarrow \mathcal{U}$ for each $i\in\{0,\dots,T-1\}$. The reach-avoid problem over a finite time horizon $T$ is to find an admissible $T$-step control policy that maximizes the probability of $x_t$ reaching the set $K$ at some time $j\leq T$ while staying in $K^{\prime}$ for all $0\leq t\leq j$. For any initial state $x_0$, we denote the reach-avoid probability associated with a given $\mu$ as $$r_{x_0}^\mu(K,K^{\prime}) = \Prb^{\mu}_{x_0}\{\exists j \in [0,T] : x_j \in K \wedge \forall i \in [0,j-1],\ x_i \in K^{\prime}\setminus K \}.$$ In the above, it is assumed that $[0,-1]=\emptyset$, which implies that the requirement on $i$ is automatically satisfied
when $x_0\in K$.

\subsection{Dynamic programming approach}
\label{sec:stochreach}

The reach-avoid probability $r_{x_0}^\mu(K,K^{\prime})$ can be equivalently formulated as an expected value objective function. In contrast to an optimal control problem with additive stage cost, $r_{x_0}^\mu(K,K^{\prime})$ is a history dependent sum-multiplicative cost function \cite{summers2010verification}:
\begin{align}
r_{x_0}^\mu(K,K^{\prime}) = \mathbb{E}_{x_0}^\mu\left [\sum^{T}_{j=0}{\left( \prod^{j-1}_{i=0}{\mathds{1}_{K^{\prime}\setminus K}(x_i)}       \right)} \mathds{1}_K(x_j)        \right],
\label{eq:raprob}
\end{align} 
where we use the notation of $\prod_{i=k}^j(\cdot)=1$ if $k > j$. Above,  $\mathds{1}_A(x)$ denotes the indicator function of a set $A\in \mathcal{B}(\Xs)$. Our objective is to find $\sup_{\mu} r^\mu_{x_0}(K,K^\prime)$ and the optimal policy achieving the supremum. The sets $K$ and $K^\prime$ can be time-varying or  stochastic \cite{summers2013stochastic} but for simplicity we assume here that they are constant. We denote the difference between the safe and target sets by $\bar{\Xs}:=K^\prime \setminus K$ to simplify the presentation of our results.

Similar to the dynamic programming approach to an optimal control problem with additive stage cost, the solution to the reach-avoid problem is characterized by a  recursion \cite{summers2010verification} as follows. Define the value functions $V_k^*:\mathcal{X}\rightarrow [0,1]$ for $k=T-1,\dots,0$ as
\begin{align}
\label{eq:DP}
\begin{split}
&V_{T}^*(x)=\mathds{1}_K(x),\\
&V_k^*(x)=\sup_{u\in \mathcal{U}}\left\{\mathds{1}_K(x)+\mathds{1}_{\bar{\Xs}}(x)\int_{\mathcal{X}}{V^*_{k+1}(y)Q(dy|x,u)}\right\}.
\end{split}
\end{align}
It can be shown that $V_0^*(x_0)=\sup_{\mu} r^\mu_{x_0}(K,K^\prime)$  \cite{summers2010verification}. Past work has  focused on approximating $V_k^*$ recursively on a discretized grid of $\bar{\Xs}$ and $\m{U}$ \cite{prandini2006stochastic, abate2007computational,summers2010verification}. Note that the DP recursion defined by \eqref{eq:DP} does not fall into the category of additive discounted cost problems. This difference, together with the finite horizon considered, yield certain approximation approaches for  MDPs with discounted additive cost function not applicable to the problem at hand.

Next, we will establish the measurability and continuity properties of the reach-avoid value functions to enable the use of a linear program  to approximate these functions.
\begin{assumption}
\label{assum:cont_kernel}
For every $x \in \m{X}, A \in \m{B}(X)$ the mapping $u \mapsto Q(A|x,u)$ is continuous.
\end{assumption}

\begin{prop}
\label{prop:u_attained}
Under Assumption \eqref{assum:cont_kernel}, at every step $k$, the supremum in \eqref{eq:DP} is attained by a measurable function $\mu_k^*:\m{X} \rightarrow \m{U}$ and the resulting $V^*_k: \Xs \rightarrow [0,1]$ is measurable. 
\label{prop:attain_pol}
\end{prop}

\begin{pf}
By induction. First, note that the indicator function $V^*_T(x)=\mathds{1}_K(x)$ is measurable. Assuming that $V^*_{k+1}$ is measurable we will show that $V^*_{k}$ is also measurable. Define $F(x,u ) = \int_{\mathcal{X}}{V^*_{k+1}(y)Q(dy|x,u)}$. Due to continuity of the map $u \mapsto Q(A|x,u)$ by Assumption \ref{assum:cont_kernel}, the map $u \mapsto F(x,u)$ is continuous for every $x$ (\cite[Fact 3.9]{nowak1985universally}). Since $\Us$ is compact, there exists a measurable function $\mu^*_k(x)$ that achieves the supremum \cite[Corollary 1]{brown1973measurable}. Furthermore, by \cite[Proposition 7.29]{bertsekas1976dynamic}, the mapping $(x,u) \mapsto F(x,u)$ is measurable. It follows that $F(x,\mu^*_k(x))$, and hence $V^*_k$, is measurable as it is composition of measurable functions. 
\end{pf}

Proposition \ref{prop:attain_pol} allows one to compute an optimal feedback policy at each stage $k$ through
\begin{align}
\nonumber
\mu^*_k(x)&=\arg\max_{u\in \mathcal{U}}\left\{\mathds{1}_K(x)+\mathds{1}_{\bar{\Xs}}(x)\int_{\mathcal{X}}{V^*_{k+1}(y)Q(dy|x,u)}\right\}\\
&=\arg\max_{u\in \mathcal{U}}\left\{\int_{\mathcal{X}}{V^*_{k+1}(y)Q(dy|x,u)}\right\}.
\label{eq:optpol}
\end{align}

For  functions $f,g:X \rightarrow \mathbb{R}$, we use $f \leq g$ to denote $f(x) \leq g(x),\; \forall x \in X$. It is easy to verify by induction that  $0\leq V^*_k \leq 1 $, for $k = T, T-1, \dots, 0$. Furthermore, due to the indicator functions in \eqref{eq:DP}, $V^*_k(x)$ are defined on disjoint regions of $\Xs$ as:
\begin{align}
{V}_k^*(x)=
\begin{cases}
1, &x\in K  \\
\max_{u\in\Us}\int_{\mathcal{X}}{{V}^*_{k+1}}(y)Q(dy|x,u), &x\in \bar{\Xs} \\
0,&x\in \Xs\setminus {K'}
\end{cases}
\label{eq:partitionVF}
\end{align}
Hence, it suffices to compute $V^*_k$  and the optimizing policy on $\bar{\Xs}$. We show that with an additional assumption on kernel $Q$, $V^*_k$ is continuous on $\bar{\Xs}$. The continuity is a desired property for approximating $V^*_k$ on $\bar{\Xs}$ using basis functions. 

\begin{assumption}
\label{assum:cont_kernel2}
For every $A \in \m{B}(\mathcal{X})$ the mapping $(x,u) \mapsto Q(A|x,u)$ is continuous.
\end{assumption}

\begin{prop}
\label{prp:cont_vf}
Under Assumption \eqref{assum:cont_kernel2}, $V^*_k(x)$ is piecewise continuous on ${\Xs}$.
\end{prop}

\begin{pf}
From  continuity of $(x,u) \mapsto Q(A|x,u)$ we conclude that the mapping $(x,u) \mapsto F(x,u)$ is continuous (\cite[Fact 3.9]{nowak1985universally}). From the Maximum Theorem \cite{sundaram1996first}, it follows that $F(x,u^*(x))$ and thus each $V^*_k(x)$, is continuous on $\bar{\Xs}$. The result follows by \eqref{eq:partitionVF}. 
\end{pf}

\subsection{Linear programming approach}
\label{sec:inflp}
Let $\m{F} := \{ f : \Xs \rightarrow \mathbb{R}, \; f \text{ is measurable} \}$. For $V\in \mathcal{F}$  define two operators $T_u, T : \m{F} \rightarrow \m{F}$ 
\begin{align}
\label{eq:bellop}
\mathcal{T}_u[V](x) &=\int_{\mathcal{X}}{V}(y)Q(dy|x,u), \\
\label{eq:bellop_u}
\mathcal{T}[V](x) &=\max_{u \in \mathcal{U}} \mathcal{T}_u[V](x).
\end{align}
Let $\nu$ be a non-negative measure supported on $\bar{\Xs}$, referred to as state-relevance measure. 
\begin{thm}
\label{prp:infinite_opt}
Suppose Assumption \ref{assum:cont_kernel} holds. For $k \in \{0,\dots,T-1 \}$, let $V^*_{k+1}$ be the value function at step $k+1$ defined in \eqref{eq:DP}. Consider the infinite dimensional linear program: 
\begin{align}
\label{eq:inflp}\tag{Inf-LP}
\inf_{V(\cdot)\in \F} \quad &\int_{\bar{\Xs}} V(x) \nu(\dx) \\
\label{eq:constraint}
\st \quad & V(x)\geq \mathcal{T}_u[V^*_{k+1}](x), \quad \forall (x,u)\in \bar{\mathcal{X}}\times\mathcal{U}.
\end{align}
(a) Any feasible solution of  \eqref{eq:inflp}  is an upper bound on the optimal reach-avoid value function $V^*_k$; (b) $V^*_k$ is a solution to this optimization problem and any other solution to \eqref{eq:inflp} is equal to $V^*_k$, $\nu$-almost everywhere on $\bar{\Xs}$.
\label{prop:1norm}
\end{thm} 
\textbf{Remark:} In order to evaluate a constraint in \eqref{eq:constraint} when an element $(x,u)$ is fixed, one has to know the value of $V^*_{k+1}$ on the set $K$. This value is  by definition of the DP in \eqref{eq:DP}  equal to one. Note that the decision variable in \ref{eq:inflp} lives in $\m{F}$, an infinite dimensional space. The objectives and constraints are linear in the decision variable but there are infinitely many constraints since $\bar{\Xs}$ and $\m{U}$ are continuous spaces. This class of problems is referred to in literature as an infinite dimensional linear program \cite{anderson1987linear,hernandez1998approximation}.

\begin{pf}
Let $J^* \in \mathbb{R}$ denote the optimal value of the objective function in \eqref{eq:inflp}. From Proposition \ref{prop:u_attained}, $V^*_k\in \F$ and is equal to the supremum over $u\in\Us$ of the right hand side of the constraint  \eqref{eq:constraint}.  Hence, for any feasible $V \in \mathcal{F}$, we have  $V(x) \geq V^*_k(x)$ for all $x \in \bar{\Xs}$ and part (a) is shown. By non-negativity of $\nu$ it follows that for any feasible $V$, $\int_{\bar{\Xs}} V(x) \nu(\dx) \geq \int_{\bar{\Xs}} V^*_k(x) \nu(\dx)$, which implies $J^*\geq \int_{\bar{\Xs}} V^*_k(x) \nu(\dx)$. On the other hand, $J^* \leq \int_{\bar{\Xs}} V^*_k(x) \nu(\dx)$ since it is the least cost among the set of feasible  functions. Hence, $J^* = \int_{\bar{\Xs}} V^*_k(x)\nu(\dx)$ and $V^*_k$ is an optimal solution. To show that any other solution to \eqref{eq:inflp} is equal to $V_k^*$ $\nu$-almost everywhere on $\bar{\Xs}$, assume  there exists a function $V^*$, optimal for \eqref{eq:inflp} that is strictly greater than $V_k^*$ on a set $A_m\in \mathcal{B}(\Xs)$ of non-zero $\nu$-measure. Since $V^*$ and $V^*_k$ are both optimal, we have that $\int_{\bar{\Xs}} V^*(x)\nu(\dx)=\int_{\bar{\Xs}} V_k^*(x)\nu(\dx)=J^*$. We can then reduce $V^*$ to the value of $V^*_k$ on $A_m$, while ensuring feasibility of $V^*$. This reduces the value of $\int_{\bar{\Xs}} V^*(x)\nu(\dx)$ below $J^*$, contradicting that $V^*$ is optimal and part (b) is shown.
\end{pf}
As shown in Theorem \ref{prp:infinite_opt}, the sequence of value functions of the stochastic reach-avoid problem derived in \eqref{eq:DP} are equivalently characterized as  solutions of a sequence of infinite dimensional linear programs. Thus, instead of the classical  space gridding approaches to solve \eqref{eq:DP}, we focus on approximating $V^*_k$ by approximating the solutions to \eqref{eq:inflp}.

\section{Approximation with a finite linear program}
\label{sec:lpcampi}
An infinite dimensional LP  is in general NP-hard \cite{anderson1987linear,hernandez1998approximation}. We approximate the solution to  \eqref{eq:inflp} by deriving a finite LP through two steps. First, we restrict the decision space to a finite dimensional subspace $\m{F}^M \subset \m{F}$. Second, we replace the infinite constraints in \eqref{eq:constraint} with a sufficiently large finite number of randomly sampled constraints. 

\subsection{Restriction to a finite dimensional function class}

Let $\mathcal{F}^M$ be a finite dimensional subspace of $\mathcal{F}$ spanned by $M$ basis elements denoted by $\{\phi_i\}_{i=1}^M$. For a fixed function $f\in \F$, consider the following semi-infinite linear program defined over functions $\sum_{i=1}^M w_i\phi_i(x) \in \m{F}^M$ with decision variable $w\in \mathbb{R}^M$:
\begin{align}
\label{eq:semiinfLP}\tag{Semi-LP}
\min_{{w}_1,\dots,{w}_{M}}\quad &\sum_{i=1}^{M} {w}_i \int_{\bar{\Xs}} \phi_i(x)\nu(\dx)\\
\label{eq:semiinfCon}
\st \quad & \sum_{i=1}^{M} {w}_i\phi_i(x)\geq\mathcal{T}_{u}[f](x), \ \forall (x,u) \in \bar{\Xs}\times \Us.
\end{align}
The above linear program has finitely many decision variables and infinitely many constraints. It is referred to as a semi-infinite linear program. 

We assume that problem \eqref{eq:semiinfLP} is feasible. Note that for a bounded  $f$, this can always be guaranteed by  including $\phi(x) = 1$ in the basis functions.  Consider the following semi-norm on $\mathcal{F}$ induced by the state-relevance measure $\nu$, $\|V\|_{1,\nu}:=  \int_{\bar{\Xs}} |V(x)|\nu(\dx)$. In the infinite dimensional linear program \eqref{eq:inflp} the choice of $\nu$ does not affect the optimal solution, as seen in Theorem \eqref{prp:infinite_opt}. For finite dimensional approximations, as will be shown in the next Lemma,  $\nu$ influences approximation accuracy in different regions of $\bar{\Xs}$.

Let $\hat{V}_f =  \sum_{i=1}^M \hat{w}_i\phi_i$ be a solution to \eqref{eq:semiinfLP} and $V^*_f \in \m{F}$ be a solution to \eqref{eq:inflp}.
\begin{lemma}
\label{lem:basis_finite}
$\hat{V}_f$ achieves the minimum of $ \big\|V-V^*_f\big\|_{1,\nu}$, over the set $\{V \in \m{F}^M, \; V \geq V^*_f\}$.
\end{lemma}
\begin{pf}
It follows from the proof of Theorem \eqref{prp:infinite_opt} that $V^*_f = \sup_u \mathcal{T}_u[f]$, $\nu$-almost everywhere. Now, a function $\hat{V} \in \m{F}^M $ is an upper bound on $V^*_f=\sup_u \mathcal{T}_u[f]$ if and only if it satisfies  constraint \eqref{eq:semiinfCon}. To show that $\hat{V}_f$ minimizes the $\nu$-norm distance to $V^*_f$, notice that for any ${V}(x)=\sum_{i=1}^{M} {w}_i\phi_i(x)$ satisfying \eqref{eq:semiinfCon} we have that
\begin{align*}
\|V-V^*_f \|_{1,\nu}=  \int_{\bar{\Xs}} |V(x)-V^*_f(x) |\nu(\dx) = &\int_{\bar{\Xs}} V(x)\nu(\dx) -\int_{\bar{\Xs}}V^*_f(x) \nu(\dx),
\end{align*}
where the second equality is due to the fact that $V$ is an upper bound of $V^*_f$.  Since $V^*_f$ is a fixed constant in the norm optimization of the lemma above, the result  follows. \end{pf}

The semi-infinite problem \eqref{eq:semiinfLP} can be used to recursively approximate $V^*_k$ using a weighted sum of basis functions. The next proposition formalizes this result. 

\begin{prop}
\label{prp:vfunction_zero}
For every $k\in\{0,\dots,T-1\}$, let $\mathcal{F}^{M_k}$ denote the span of a fixed set of $M_k$  basis elements $\{\phi^k_i\}_{i=1}^{M_k}$ where each $\phi_i^k\in\m{F}$. Start with the known value function $V^*_{T}$ and recursively construct $\hat{V}_k(x)=\sum_{i=1}^{M_k} {\hat{w}^k_i}\phi^k_i(x)$ where ${\hat{w}^k_i}$ is the solution to  \eqref{eq:semiinfLP} obtained by substituting $f=\hat{V}_{k+1}$ in \eqref{eq:semiinfLP}. Then,\\ (a) Each $\hat{V}_k$ is also a solution to the  problem:
\begin{align}
\label{eq:minerror}
\min_{V(\cdot)\in \F^{M_k}} \quad & \big\|V-V^*_k\big\|_{1,\nu} \\
\label{eq:recursive_ub}
\st \quad & V(x)\geq \mathcal{T}_u[\hat{V}_{k+1}](x), \quad \forall (x,u)\in \bar{\Xs}\times\mathcal{U}.
\end{align}
(b) $\hat{V}_k(x)\geq V_k^*(x)$ for all $x\in\bar{\Xs}$ and $k = 0, \dots, T-1$.
\end{prop}
\begin{pf}
First, we prove part (b) by induction. Note that at step $T-1$ the results above hold as a direct consequence of Lemma \eqref{lem:basis_finite}. Now, suppose at time step $k$, $\hat{V}_k(x)\geq V_k^*(x)$. From monotonicity of the operator $\m{T}_u$ \cite{summers2010verification}, it follows that $\m{T}_u [\hat{V}_k](x) \geq \m{T}_u [V^*_k](x)$. By constraint  \eqref{eq:recursive_ub}, it follows that $\hat{V}_{k-1}(x) \geq \m{T}_u [\hat{V}_k](x) \geq \m{T}_u [V^*_k](x) = V^*_{k-1}(x)$, where the last equality is due to the definition of $V^*_k$ in \eqref{eq:partitionVF}. Hence, part (b) is proven. To prove part (a), consider that $\hat{V}_k$ is the solution for \eqref{eq:semiinfLP} with $f=\hat{V}_{k+1}$ which implies that $\hat{V}_k(x)\geq \mathcal{T}_u[\hat{V}_{k+1}](x),\forall (x,u)\in \bar{\Xs}\times\mathcal{U}$ and it thus satisfies \eqref{eq:recursive_ub}. Being a solution to \eqref{eq:semiinfLP} also implies that $\hat{V}_k$ achieves the minimum of $\| V-{V}^*_{k+1}\|_{1,\nu}$ over the set $\{V \in \m{F}^{M_k}\}$. The cost function $\|V-V^*_k \|_{1,\nu}$ expands to $\|V-V^*_k \|_{1,\nu}=\int_{\bar{\Xs}} |V(x)-V^*_k(x) |\nu(\dx) = \int_{\bar{\Xs}} V(x)\nu(\dx) -\int_{\bar{\Xs}}V^*_k(x) \nu(\dx)$ where the last step follows from part (b) proven above. Hence minimizing $\| V-{V}^*_{k}\|_{1,\nu}$ is equivalent to minimizing $\| V-{V}^*_{k+1}\|_{1,\nu}$ since ${V}^*_{k+1}$ and $V^*_k(x)$ are fixed.
\end{pf}
Notice that since $\nu$ is a probability measure, it readily follows that the approximation error satisfies $ \big\|\hat{V}_k-V^*_k\big\|_{1,\nu} \leq \big\|\hat{V}_k-V^*_k\big\|_{\infty}$.

The above proposition shows that by restricting the decision space of the infinite dimensional linear program, we obtain an upper bound to the reach-avoid value functions $V^*_k$, at every step $k$, which is also the least upper bound in the space spanned by the basis functions subject to constraint \eqref{eq:recursive_ub}. To compute a constraint in \eqref{eq:recursive_ub} for a fixed pair $(x,u)$, one has to use of the fact that the ($k+1$)-approximate value function $\hat{V}_{k+1}$, can be set to 1 on the target set $K$ since the reach-avoid probability is by definition equal to 1 on  set $K$ (observe the reach-avoid statement in Section \ref{sec:stochreach_DP}).


\subsection{Restriction to a finite number of constraints}
\label{sec:finite_constr}
A semi-infinite linear program, such as \eqref{eq:semiinfLP} is in general NP-hard  \cite{bertsimas2011theory,ben2002robust,hettich1993semi} due to existence of infinitely many constraints, one for each state-action pair $(x,u) \in \bar{\Xs} \times \Us$. One way to approximate the solution is to select a finite set of points from $\bar{\Xs}\times \Us$ to impose the constraints on. One can then use generalization results from sampled convex programs \cite{calafiore2006scenario, campi2009scenario} to quantify the near-feasibility of the solution obtained from constraint sampling.  

Note that there are several methods to approximate semi-infinite LPs. An alternative approach that could be used in solving the robust or semi-infinite LPs is the constraint generation technique \cite{patrascu2002greedy,guestrin2003efficient}. In constraint generation, motivated by the fact that only finitely many constraints are active at the optimal solution, one tries to iteratively add constraints to the problem and reduce the gap between the approximate and the optimal solution. Such algorithms may be preferred due to their convergence properties. Here, we consider a one-shot sampling approach with probabilist guarantees. 


Let $S:=\{(x^i,u^i)\}_{i=1}^{N}$ denote a set of $N\in\mathbb{N}$ elements in $\bar{\Xs}\times\Us$. For a fixed function $f\in \F$, consider the following finite LP defined over functions $ \sum_{i=1}^M w_i \phi_i(x) \in \F^M$:
\begin{align}
\label{eq:finiteLP}\tag{Fin-LP}
\begin{split}
\min_{{w}_1,\dots,{w}_{M}}\quad &\sum_{i=1}^{M} {w}_i \int_{\bar{\Xs}} \phi_i(x)\nu(\dx)\\
\st \quad & \sum_{i=1}^{M} {w}_i\phi_i(x)\geq\mathcal{T}_{u}[f](x), \ \forall (x,u) \in S
\end{split}
\end{align}
Since the objective and constraints are linear in the decision variable, the sampling theorem \cite[Theorem 1]{campi2009scenario} applies. This theorem states that for any chosen pair of confidence and constraint violation probabilities (denoted by $\eps,\beta$), one only needs to consider a finite number of constraints (denoted by $N$) that can be readily computed from $\eps,\beta$ and the decision variables. It is a probabilistic feasibility guarantee similar to Probably-Approximately-Correct (PAC) statements, tailored to convex problems \cite{margellos2015connection}.
\begin{lemma}
\label{thm:campiVF}
Assume that for any  $S \subset \bar{\Xs}\times U$, the feasible region of \eqref{eq:finiteLP} is non-empty and the optimizer is unique. Choose the violation and confidence levels $\eps,\beta \in (0,1)$. Construct a set of samples $S$ by drawing $N$ independent points from $\bar{\Xs}\times \Us$ identically distributed  according to a probability measure on $\bar{\Xs}\times \Us$ denoted by $\Prb_{\bar{\Xs} \times \Us}$. Choose $N$ such that
\begin{align*}
N\geq \frac{2}{\varepsilon}\left(M+\ln\left(\frac{1}{\beta}\right)\right).
\end{align*}
Let $\tilde{w}^S$ be the  sample dependent optimizer in \eqref{eq:finiteLP}, and $\tilde{V}^S(x)=\sum_{i=1}^{M}{\tilde{w}^S_i}\phi_i(x)$. Then, 
\begin{align}
\label{eq:mstageGua}	
\Prb_{\bar{\Xs} \times \Us} (\tilde{V}^S(x)< \mathcal{T}_u[f](x))\leq \varepsilon
\end{align}
with confidence $1-\beta$,  with respect to the product measure $(\Prb_{\bar{\Xs} \times \Us})^N$.
\end{lemma}
The probabilistic expression in \eqref{eq:mstageGua} is referred to as violation of $\tilde{V}^S$ \cite{calafiore2006scenario, campi2009scenario}. Note that $\tilde{V}^S$ is a function of the $N$ sample realizations. As such, it can only be bounded  to an $\epsilon$-level with a confidence with respect to the product measure $(\Prb_{\bar{\Xs} \times \Us})^N$. 

We can recursively construct $\tilde{V}_{k} = \sum_{i=1}^{M_k}\tilde{w}^k_i \phi^k_i$ by solving \eqref{eq:finiteLP} using $f=\tilde{V}_{k+1}$ and a number $N_k(\eps_k, \beta_k, M_k)$, of samples.  It follows that with probability greater than $1-\beta_k$, the violation of $\tilde{V}_k$ is at most $\varepsilon_k$. Consequently, the approximation functions $\tilde{V}_k$ are probabilistic upper bounds on the value functions $V^*_k$, in contrast to the guaranteed upper bounds provided in Proposition \eqref{prp:vfunction_zero}.  


Note that one can also formulate the recursion of approximate linear programs as a single linear program by summing up the individual cost functions and sampling the constraints collectively. The authors in \cite{kariotoglou2016computational} analyze the change in probabilistic guarantees in this case and the advantages in computational complexity that one gets by treating the problems in a recursion as opposed to a single problem. We also note that this recursive approximation technique can be extended to arbitrarily large horizons with a linear increase in the number of decision variables. However, approximations may get progressively worse. 

To evaluate the accuracy of $\tilde{V}_k$, ideally, we  aim to find bounds on $\| \tilde{V}_k - V^*_k \|_{1,\nu}$ as a function of  $\| \hat{V}_k - V^*_k\|_{1,\nu}$, where $\hat{V}_k$ is computed according to Proposition \eqref{prp:vfunction_zero} and then determine the number of basis functions required to bound $\| \hat{V}_k - V^*_k\|_{1,\nu}$ to a given accuracy. As for the first problem, for a given accuracy in the objective function of a sampled convex program, one may be able to get bounds on the number of samples  depending on the Lipschitz constant of the objective function \cite{esfahani2013performance}. As for the second problem, the number of basis required to bound  $\| \hat{V}_k - V^*_k\|$ is heavily dependent on the basis function choice and Lipschitz continuity properties of the objective function.  For a technical discussion on the general problem of quantifying the error between an infinite dimensional LP  and approximations based on finite dimensional restrictions we refer readers to \cite{hernandez1998approximation,esfahani2013performance}. Note that the bounds derived in these works are only valid for discounted MDPs with additive stage-cost. Extension of the results to stochastic reach-avoid problems is more challenging due to the sum-multiplicative cost function structure and the lack of a discount factor. 

In the remainder of the paper we will evaluate the computational tractability and accuracy of  \eqref{eq:finiteLP}  in estimating reach-avoid value functions for a general class of MDPs.


\section{Radial basis functions for MDPs with Gaussian mixture kernels}
\label{sec:implement}
For a general class of MDPs modeled by Gaussian mixture kernels \cite{khansari2011learning} we propose using Gaussian radial basis functions (GRBFs) for approximating the reach-avoid value functions. Through this choice, the constraint in \eqref{eq:finiteLP} involving the integration $\m{T}_u[f]$ can be found in closed form. Moreover, it is known that radial basis functions are a sufficiently rich function class to approximate continuous functions \cite{hartman1990layered,sandberg2001gaussian,park1991universal,cybenko1989approximation}. In fact, the authors in \cite{kveton2006learning} also propose an algorithm for learning basis functions in the context of approximate linear programming using mean-parametrized GRBFs.

\subsection{Basis function choice}
\label{sec:rbasis}
To apply GRBFs in the reach-avoid framework, we consider the following problem data: 
\begin{enumerate}
\item The  kernel $Q$ is a Gaussian mixture kernel $\sum_{j=1}^J \alpha_j\mathcal{N}(\mu_j,\Sigma_j)$ with  diagonal covariance matrices $\Sigma_j$, means $\mu_j$ and weights $\alpha_j$ such that $\sum_{j=1}^J\alpha_j=1$ for a finite $J\in\mathbb{N}_+$.
\item The target and safe sets $K$ and $K^\prime$ are unions of disjoint hyper-rectangle sets, i.e. $K=\bigcup_{p=1}^P K_p=\bigcup_{p=1}^P \parens*{\prod_{l=1}^n [a^p_l,b^p_l]}$ and $K^\prime=\bigcup_{m=1}^M K^\prime_m=\bigcup_{m=1}^M \parens*{\prod_{l=1}^n [c^m_l,d^m_l]}$ for finite $P,M \in\mathbb{N}_+$ with $n=\dim\parens*{\Xs}$ and $a^p,b^p,c^m,d^m\in \mathbb{R}^n$, $\forall p, m$.
\end{enumerate}
The above restrictions apply to a large class of MDPs. For example, the kernel of general non-linear systems subject to additive Gaussian mixture noise is a Gaussian mixture kernel. Moreover, in several problems, the state and input constraints are decoupled in different dimensions resulting in disjoint hyper-rectangles as constraint sets. It should be noted that whenever the safe and target sets are polytopic and cannot be written as unions of disjoint hyper-rectangles, one
can approximate them as such to arbitrary accuracy using the methods in \cite{bemporad2004inner}. For a fixed approximation accuracy such algorithms are of polynomial complexity with respect to the dimension of the space. Approximations of more general sets with polytopes within a predefined accuracy is a much harder problem and algorithms may scale exponentially in the dimension of the problem \cite{bronstein2008approximation}.

For each time step $k$, let $\F^{M_k}$ denote the span of a set of GRBFs $\{\phi^k_i\}_{i=1}^{M_k}:\mathbb{R}^n\rightarrow \mathbb{R}$: 
\begin{align}
\small{
\phi^k_i(x) = \prod^n_{l=1}{\frac{1}{\sqrt{2\pi s^{k}_{i,l}}}\exp\parens*{-\frac{1}{2}\frac{(x_l-c^{k}_{i,l})^2}{s^{k}_{i,l}}}},
\label{eq:basisf}}
\end{align}
where $\{c^{k}_{i,l}\}_{l=1}^n \in \mathbb{R}, \{s^{k}_{i,l}\}_{l=1}^n \in \mathbb{R}_+$ are the centers and the variances, respectively, of the GRBF. The class of GRBFs is closed with respect to multiplication \cite[Section 2]{hartman1990layered}.  In particular, let $f^1=\sum_{i=1}^{M_k}{w^1_i \phi^k_i}$, $f^2=\sum_{j=1}^{M_k}{w^2_j \phi^k_j}$. Then, $f^1f^2=\sum_{i=1}^{M_k}\sum_{j=1}^{M_k}{w^1_iw^2_j \tilde{\phi}^k_{ij}}$, where the centers and variances of the bases $\tilde{\phi}^k_{ij}$ are explicit functions of those of $\phi^k_i, \phi^k_j$. 

Integrating the proposed GRBFs over a union of hyper-rectangles decomposes into one dimensional integrals of Gaussian functions. In particular, let $\tilde{V}_k(x)=\sum_{i=1}^{M_k} \tilde{w}^k_i\phi_i^k(x)$ denote the approximate value function at time $k$ and $A=\bigcup_{d=1}^D  \bracks*{[a^d_1,b^d_1] \times \allowbreak\cdots \times\allowbreak [a^d_n,b^d_n]}$, a finite union of hyper-rectangles. The integral of $\tilde{V}_k$ over $A$ after some algebra reduces to
\begin{align}
\begin{split}
\int_{A}{\tilde{V}_k(x)\nu(\dx)} &= \sum_{d=1}^D \sum_{i=1}^{M_k} \tilde{w}^k_i{\prod_{l=1}^n{\left(\frac{1}{2}\erf\left(\frac{b^d_l-c^{k}_{i,l}}{\sqrt{2s^{k}_{i,l}}}\right)-\frac{1}{2}\erf\left(\frac{a^d_l-c^{k}_{i,l}}{\sqrt{2 s^{k}_{i,l}}}\right)\right)}},
\end{split}
\label{eq:integral_rbf}
\end{align}
where $\nu$ is assumed to be uniform product measure on each dimension $d$ and $\erf$ denotes the error function defined as $\erf(x)=\tfrac{2}{\sqrt{\pi}}\int_0^{x}\exp\parens*{-t^2}\text{dt}$. 

Due to the decomposition of the reach-avoid value functions on the sets $ K=\cup_{p=1}^P K_p$ and $\bar{\Xs}=(K'=\cup_{m=1}^M K'_m)\setminus K$ as stated in \eqref{eq:partitionVF},  $\m{T}_u[\tilde{V}_k]$ in \eqref{eq:bellop} is equivalent to
\begin{align}
\int_{\mathcal{X}}{\tilde{V}_{k}}(y)Q(dy|x,u)&= \sum_{m=1}^M\sum_{i=1}^{M_{k}}\tilde{w}^k_i\int_{K^\prime_m}\phi^k_i(y)Q(dy|x,u)+\sum_{p=1}^P\int_{K_p}Q(dy|x,u).
\label{eq:integral_space}
\end{align}
Since  a Gaussian mixture kernel $Q$ can be written as a GRBF, every term inside the integral above is a product of GRBFs. Hence, it is a GRBF with known centers and variances.  The integrals over $K_p$ and $K^\prime_m$ can thus be computed using \eqref{eq:integral_rbf} at a sampled point $(x^s,u^s)$. 
\subsection{Recursive value function and policy approximation}
\label{sec:recur_approx}
We summarize the method to approximate the reach-avoid value function in Algorithm \ref{algo:vfalgo}. The design choices include the number of basis functions, their centers and variances, the sample violation and confidence bounds in Lemma \ref{thm:campiVF} and the state-relevance weights. The number of basis functions is problem dependent and in our case studies, we use trial and error to fix this number. We choose the centers of the GRBFs by sampling them from a  uniform probability measure supported on $\bar{\Xs}$. We sample the variances from a uniform measure supported on a  bounded set that depends on problem data. Note that the method is still applicable if centers and variances are not sampled but set in another way, for example using neural network training or trial and error. Typically, $\varepsilon$ and $\beta$ are chosen to be close to 0 to enhance the feasibility guarantees of Lemma \ref{thm:campiVF} at the expense of more constraints in \eqref{eq:finiteLP}. Furthermore, we choose the state-relevance measure $\nu$ as a uniform product measure on the space $\bar{\Xs}$ to use the analytic integration in \eqref{eq:integral_rbf}. This corresponds to equal weighting on potential errors on different state-space regions. 

Given the approximate value functions, we compute the so-called greedy control policy:
\begin{align}
\begin{split}
\tilde{\mu}_k(x)&=\arg\max_{u\in \mathcal{U}}\int_{\Xs}{\tilde{V}_{k+1}(y)Q(dy|x,u)}.
\end{split}
\label{eq:approx_pol}
\end{align}
The optimization problem in \eqref{eq:approx_pol} is non-convex. However, the cost function is smooth with respect to $u$ for a fixed $x\in\bar{\Xs}$, the gradient and Hessian information can be analytically obtained using the $\erf$ function and the decision space $\Us$ is typically low dimensional (in most mechanical systems for example, $\dim{\Us}\leq\dim{\Xs})$. Thus, a locally optimal solution can be obtained efficiently using off-the-shelf optimization solvers.

\begin{algorithm}[t]
\caption{linear programming based reach-avoid value function approximation}
\label{algo:vfalgo}
\begin{algorithmic}
\STATE \textbf{Input Data:} 
\vspace{-.3cm}
\begin{itemize}
 \setlength\itemsep{.04em}
\item State and control spaces $\bar{\Xs}\times \Us$, reach-avoid time horizon $T$.
\item Target and safe sets $K$ and $K^\prime$, written as unions of disjoint hyper-rectangles.
\item Centers and variances of the MDP Gaussian mixture kernel $Q$.
\end{itemize}
\STATE \textbf{Design parameters:} 
\vspace{-.3cm}
\begin{itemize}
 \setlength\itemsep{.04em}
\item Number of basis functions $\bracks*{M_k}_{k=0}^{T-1}$.
\item Violation and confidence levels $\bracks*{\varepsilon_i}_{i=0}^{T-1}$, $\bracks*{1-\beta_i}_{i=0}^{T-1}$,  probability measure $\Prb_{\bar{\Xs}\times \Us}$.
\item Probability measure of centers and variances for the basis functions $\{\phi^k_i\}_{i=1}^{M_k}$.
\item State-relevance measure $\nu$ decomposed as a product measure on the state space.
\end{itemize}

\STATE Initialize $\tilde{V}_T(x)\leftarrow \mathds{1}_{K}(x)$.
\FOR {$k=T-1, T-2$, $\dots$, 0}
\STATE Construct $\mathcal{F}^{M_k}$ by sampling $M_k$ centers $\bracks*{c_i}_{i=1}^{M_k}$ and variances $\bracks*{s_i}_{i=1}^{M_k}$ according to the chosen probability measures.
\STATE Sample $N(\varepsilon_k,\beta_k,M_k)$ pairs $(x^s,u^s)$ from $\bar{\Xs}\times \Us$ using the measure $\Prb_{\bar{\Xs}\times \Us}$.
\FORALL {$\parens*{x^s,u^s}$}
	\STATE Evaluate $\mathcal{T}_{u^s}[\tilde{V}_{k+1}]\parens*{x^s}$ using \eqref{eq:integral_space}.
\ENDFOR
\STATE Solve the finite LP in \eqref{eq:finiteLP} to obtain $\tilde{w}^k=(\tilde{w}^k_1,\dots,\tilde{w}^k_{M_k})$.
\STATE Set the approximated value function on $\bar{\Xs}$ to $\tilde{V}_k (x) = \sum_{i=1}^{M_k}\tilde{w}^k_i \phi^k_i (x)$.
\ENDFOR
\end{algorithmic}
\end{algorithm}

\section{Numerical case studies}
\label{sec:evaluate}

We develop and solve a series of benchmark problems and evaluate our approximate solutions in two ways. First, we compute the closed-loop  empirical reach-avoid policy by applying the approximated control input  obtained from \eqref{eq:approx_pol}. Second, we use scalable  alternative approaches to approximate the benchmark reach-avoid problems. To this end, we consider three reach-avoid problems that differ in structure and complexity. The first two examples are academic and illustrate the scalability and accuracy of the approach. The last example is a practical problem, where the approach was also implemented on a miniature race-car testbed. Throughout, we refer to our approach as the ADP  approach. All computations were carried out on an Intel Core i7 Q820 CPU clocked at 1.73 GHz with 16GB of RAM memory, using IBM's CPLEX optimization toolkit in its default settings.

\subsection{\textbf{Example 1}}
\label{sec:origin_lqg}
We consider linear systems with additive Gaussian noise, $x_{k+1}= x_k + u_k +\omega_k,$
where $x_k \in \Xs= \mathbb{R}^n$, $u_k\in \Us= [-0.1,0.1]^n$ and $\omega_k$ is distributed as a Gaussian random variable $\omega_k\sim\mathcal{N}\parens*{\mathbf{0}_{n\times 1},\Sigma}$ with diagonal covariance matrix. We consider a target set $K=\sqbracks*{-0.1,0.1}^n$ centered at the origin and a safe set $K^\prime=\sqbracks*{-1,1}^n$ (see Figure \ref{fig:LQRfig} for a 2D illustration). The objective is to reach the target set while staying in the safe set over a horizon of $T=5$ steps. We  approximated the value function using Algorithm \ref{algo:vfalgo} for a range of system dimensions  $\dim\parens*{\XUs}=4,6,8$,  to analyze scalability and accuracy of the LP-based reach-avoid solution in a benchmark problem that scales up in a straightforward way.  

\begin{figure}[t]
\centering
\begin{minipage}[t]{0.44\linewidth}
\centering
\vspace{0pt}
\includegraphics[width=7cm, height=4.5cm]{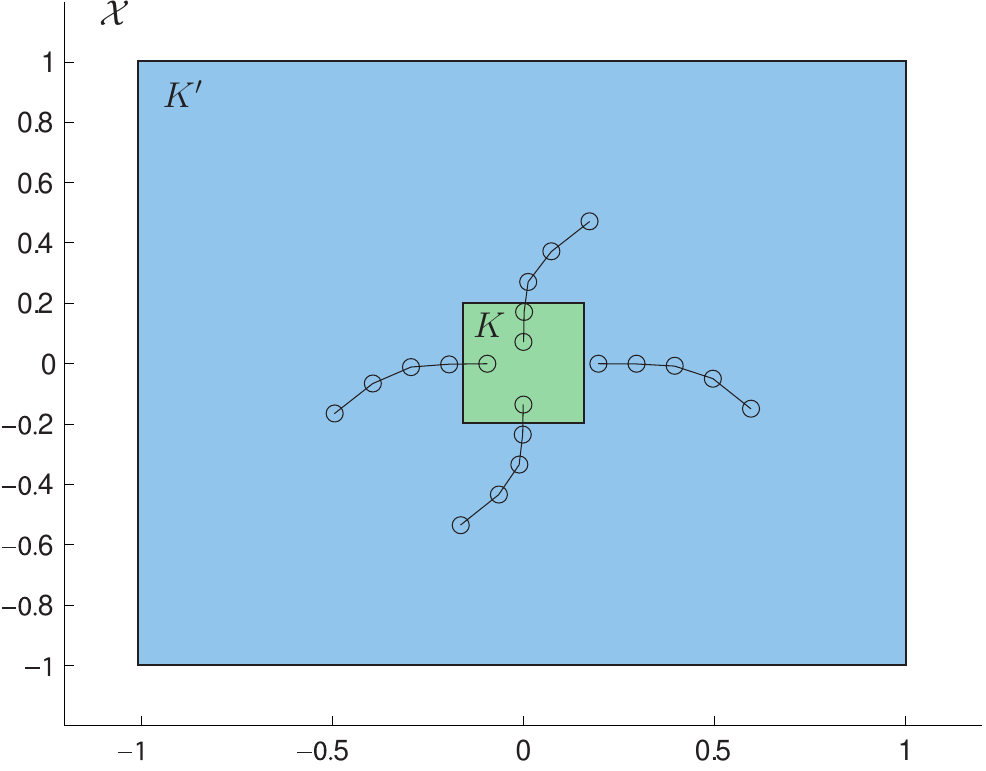}%
\captionof{figure}{2D depiction of safe and target sets and sample trajectories.}
\label{fig:LQRfig}%
\end{minipage}\hfill
\begin{minipage}[t]{0.5\linewidth}
\centering
\vspace{0pt}
\begin{tabular}{llll}
\textbf{$\dim\parens*{\XUs}$}&4D&6D&8D \\
\hline
$M_k$ &100 & 500 & 1000 \\
$N_k$ & 4184 &  20184 & 40184\\
$\varepsilon_k$& 0.05 &0.05& 0.05 \\
$1-\beta_k$& 0.99 &0.99& 0.99 \\
$\|\tilde{V}_0-V_{\text{ADP}}\|$ & 0.0692 &  0.104&    0.224\\
Construction  (sec)&4&85&450\\
LP solution (sec) &2&50&520\\
Memory (MB) & 3.2& 80&320\\
\hline
\end{tabular}
\captionof{table}{Parameters and properties of the value function approximation scheme.}
\label{tab:compareLQR}
\end{minipage}
\vspace*{-.3cm}
\end{figure}

The transition kernel of the considered linear system is Gaussian $x_{k+1}\sim \mathcal{N}(x_k+u_k,\Sigma)$. The sets $K$ and $K^\prime$ are hyper-rectangles. Thus, the GRBF framework applies. We chose $100$, $500$ and $1000$ GRBF elements for the reach-avoid problems of $\dim(\XUs)=4,6,8$, respectively (Table \ref{tab:compareLQR}). We used uniform measures supported on $\bar{\Xs}$ and $[0.02,0.095]^n$  to sample the GRBFs' centers and variances, respectively. The violation and confidence levels for every $k\in \{0,\dots,4\}$ were set to $\varepsilon_k=0.05$, $1-\beta_k=0.99$ and the measure $\Prb_{\bar{\Xs}\times\Us}$ required to generate samples from $\bar{\Xs}\times \Us$ was chosen to be uniform. Since there is no reason to favor some states more than others, we also chose $\nu$ as a uniform measure, supported on $\bar{\Xs}$. Following Algorithm \ref{algo:vfalgo} we obtain a sequence of approximate value functions $\{\tilde{V}_k\}_{k=0}^{4}$. 

To evaluate the performance of the approximation, we sampled 100 initial conditions $x_0$, uniformly from $\bar{\Xs}$. For each initial condition we generated 100  noise trajectories $\{\omega_k\}_{k=0}^{T-1}$. We computed the policy along the resulting state trajectory using \eqref{eq:approx_pol}. We then counted the number of trajectories that successfully completed the reach-avoid objective, i.e. reach $K$ without leaving $K^\prime$ in $T$ steps. In Table \ref{tab:compareLQR} we denote by $\|\tilde{V}_0-V_\text{ADP}\|$ the mean absolute difference between the empirical success denoted by $V_\text{ADP}$, and the predicted performance $\tilde{V}_0$, evaluated over the considered initial conditions. The memory and computation times reported correspond to constructing and solving each LP. 

Since the system is linear, the noise is Gaussian and the target and safe sets are symmetric and centered around the origin, we can use  the so-called Linear Quadratic Gaussian (LQG) controller \cite{bertsekas1995dynamic}. This controller has the objective to drive the states close to the origin while ensuring the energy of the input is minimized. The closed-form optimal policy for the LQG problem can be easily computed \cite{bertsekas1995dynamic}. As such, by properly tuning the corresponding weights of the states and inputs in the LQG objective based on the target and constraint sets, we can  heuristically achieve the reach-avoid objective. This is further explained below.

\begin{figure}
\centering
\begin{subfigure}[b]{0.45\textwidth}
\includegraphics[width=\textwidth]{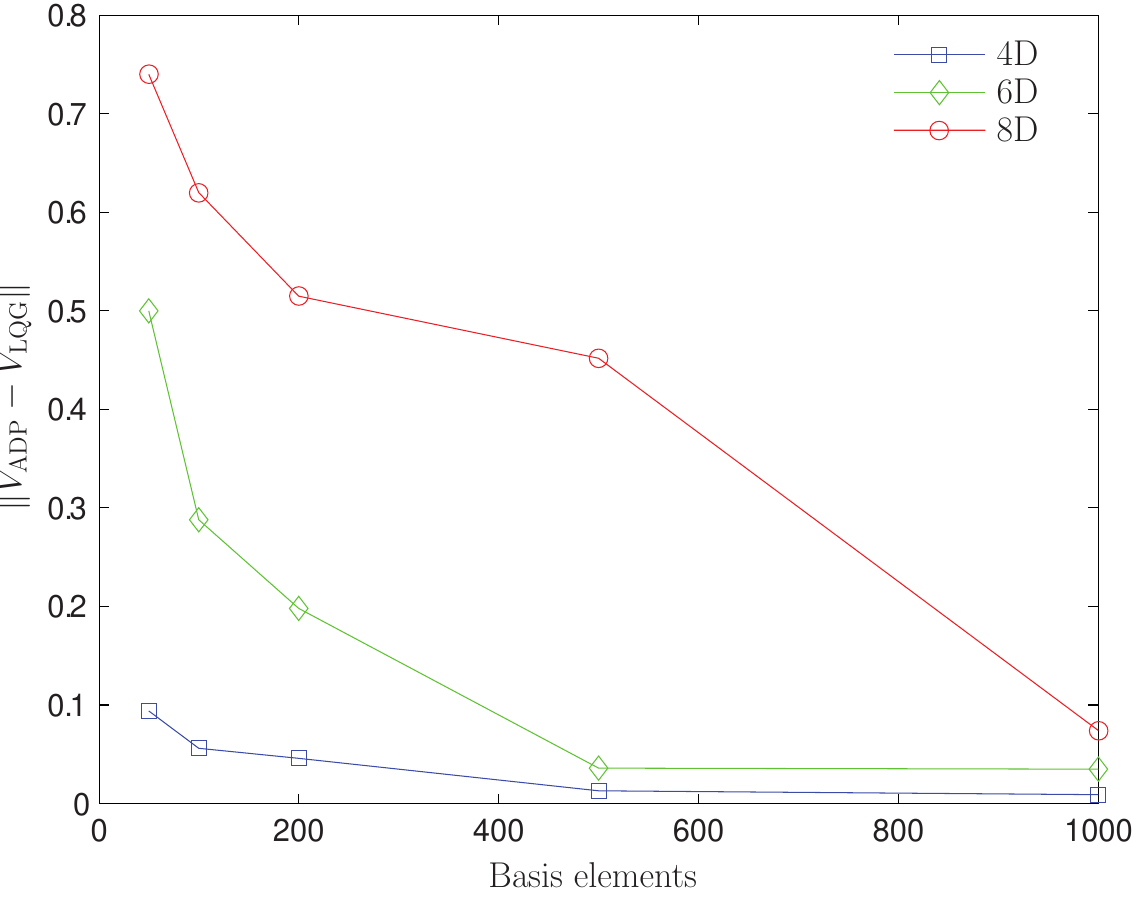}
\caption{Mean absolute difference between the empirical reach-avoid probabilities achieved by the ADP ($V_\text{ADP}$) and LQG ($V_\text{LQG}$) policies as a function basis number. {For each case of $\dim\parens*{\XUs}=4,6,8$, the number of samples is kept constant over the different number of basis elements $50,100,200,500,1000$ and is equal to the numbers reported in Table \ref{tab:compareLQR}.}}
\label{fig:tiger}
\end{subfigure}
\hspace{.3cm}
\begin{subfigure}[b]{0.51\textwidth}
\includegraphics[width=\textwidth]{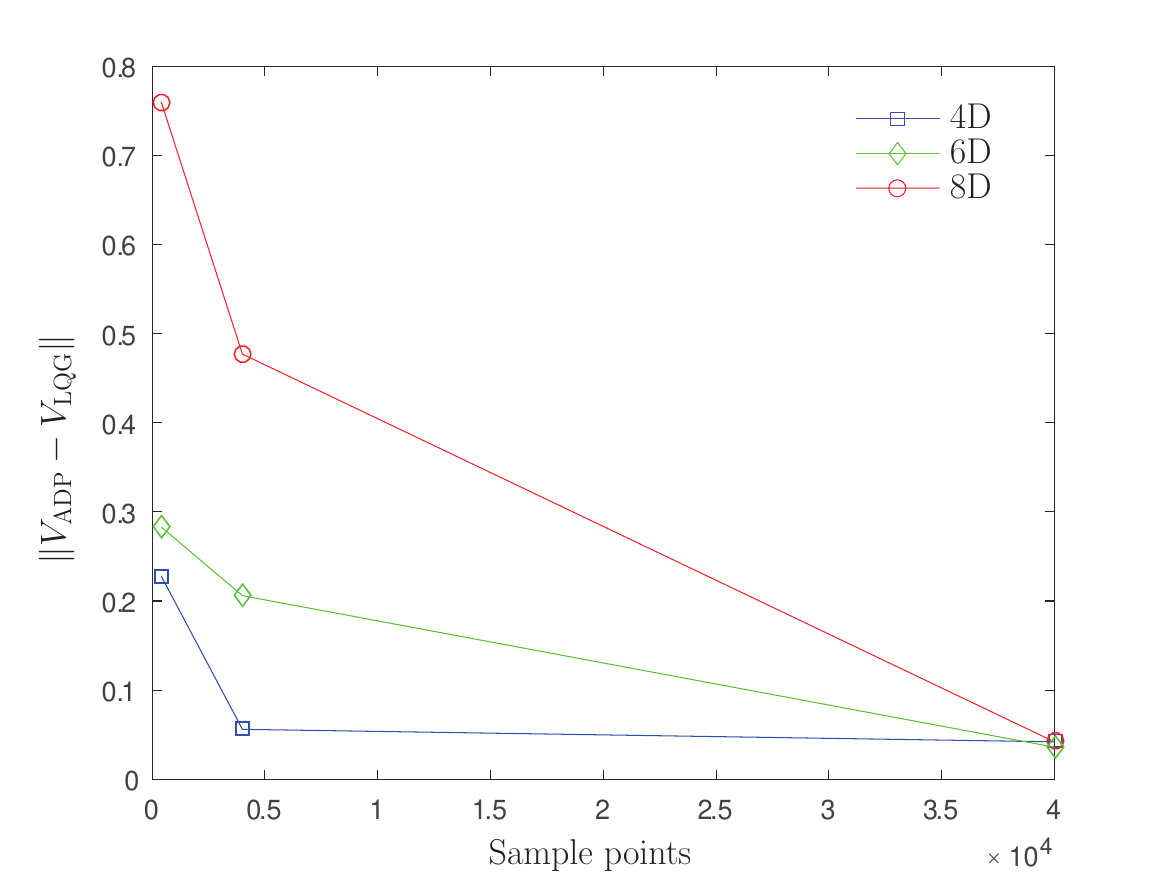}
\caption{Mean absolute difference between the empirical reach-avoid probabilities achieved by the ADP ($V_\text{ADP}$) and LQG ($V_\text{LQG}$) policies as a function of number of samples. {For each case of $\dim\parens*{\XUs}=4,6,8$, the number of basis elements is kept constant over the different number of sample elements $400,4000,40000$ and is equal to the numbers reported in Table \ref{tab:compareLQR}}.}
\label{fig:gull}
\end{subfigure}
\caption{Example 1 - performance of the algorithm as a function of parameters.}\label{fig:animals}
\end{figure}

The LQG problem for a linear  stochastic system $x_{k+1} = Ax_k + Bu_k + \omega_k$, as the one considered above, is defined by an expected value quadratic cost function: 
\vspace*{-.2cm}
\begin{align*}
\min_{\bracks*{u_k}_{k=0}^{T-1}} \mathbb{E}^\mu_{x_0}(\sum_{k=0}^{T-1}x_k^\top Q x_k + u_k^\top R u_k) + x_T^\top Q x_T.
\end{align*} 
Above, $Q\in \m{S}^n_+$ and $R \in \m{S}^m_{++}$, where $\m{S}^n_+$ and $\m{S}^m_{++}$ denote the set of $n\times n$ positive semidefinite and $m\times m$ positive definite matrices, respectively. We choose $Q$ and $R$ to correspond to the largest ellipsoids  inscribed in $K$ and $\Us$, respectively. Through this choice the level sets of the LQG cost function proportionally correspond to the size of the target and control constraint sets. Intuitively, the penalization of states through the quadratic cost $Q$ drives the state to the origin. The penalization of the input does not guarantee feasibility of the input constraints. Therefore, we project the LQG control on the feasible set $\Us$. Using the same initial conditions and noise trajectories as those used with the ADP controller above, we simulated the performance of the LQG controller.  We counted the number of trajectories that reach $K$ without leaving $K^\prime$ over the horizon of $T=5$ steps. 

Figure \ref{fig:tiger} shows the mean over the initial conditions of the absolute difference between $V_\text{LQG}$ and $V_\text{ADP}$ as a function of number of basis functions. We observe a trend of increasing accuracy with increasing number of basis functions. Figure \ref{fig:gull} shows the same metric but as a function of the total number of sample pairs from $\XUs$ for a fixed number of basis functions. Changing the  number of samples $N$, affects the violation level $\varepsilon_k$ (assuming constant $\beta_k$) and the approximation quality  seems to improve with increasing $N$. In Table \ref{tab:para_tableLQR_samp},  we observe a trade-off between accuracy and computational time for the 6D problem varying the number of samples; the result is analogous in the 4D and 8D problems.

\begin{center}
\begin{table}[!h]
\centering
\begin{tabular}{lllll}
N &$\|{V}_\text{ADP}-V_{\text{LQG}}\|$&Construction (sec)&LP solution (sec)& Memory (MB)\\
\hline
400 &0.283&2.20 & 3.57&1.60\\
4000 &0.206&17.0 & 97.0&16.0\\
40000 &0.036&170 &162&160\\
\hline
\end{tabular}
\caption{Accuracy and computation time as a function of number of sampled points in $\dim\parens*{\XUs}=6$, with $M_k=500$ and $1-\beta_k=0.99$.}
\label{tab:para_tableLQR_samp}
\end{table}
\end{center}

\begin{figure}[!t]
\begin{minipage}[t]{0.4\linewidth}
\centering
\vspace{-3pt}
\includegraphics[width=6.3cm, height=4.3cm]{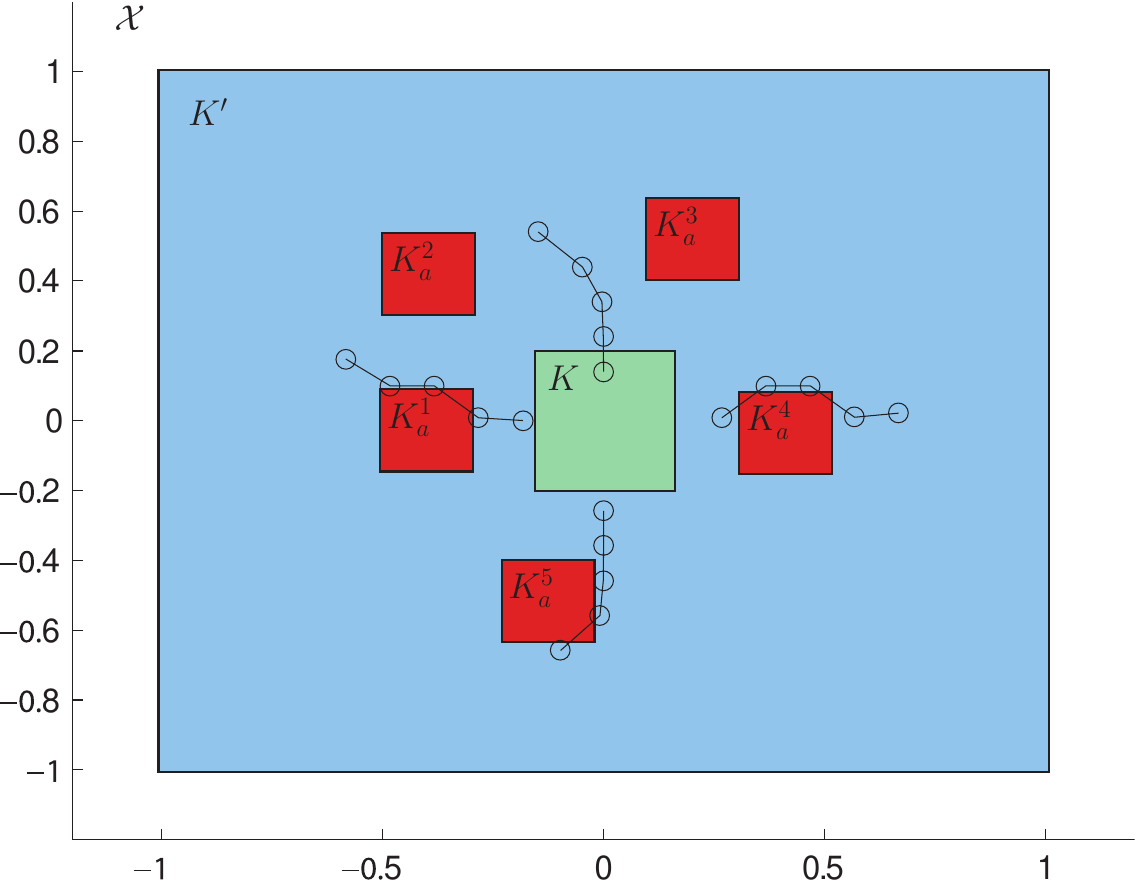}%
\caption{Example 2 - 2D depiction of obstacles and sample trajectories.}%
\label{fig:MIQPfig}%
\end{minipage}\hfill
\begin{minipage}[t]{0.5\linewidth}
\centering
\vspace{-3pt}
\begin{tabular}{llll}
\textbf{$\dim\parens*{\XUs}$}&4D&6D&8D \\
\hline
$M_k$ &100 & 500 & 1000 \\
$N_k$ & 4184 &  20184 & 40184\\
$\varepsilon_k$& 0.05 &0.05& 0.05 \\
$1-\beta_k$& 0.99 &0.99& 0.99 \\
$\|\tilde{V}_0-V_{\text{ADP}}\|$ & 0.095 &  0.118&    0.191\\
Construction (sec) &4.20&130&671\\
LP solution  (sec) &3.2&80&700\\
Memory (MB) & 3.20& 80.0&320\\
\hline
\end{tabular}
\captionof{table}{Parameters and properties of the value function approximation scheme.}
\label{tab:compareMIQP}
\end{minipage}
\vspace*{-.5cm}
\end{figure}

\subsection{\textbf{Example 2}}
\label{sec:origin_constr}

We consider the same linear dynamical system $x_{k+1} = x_k + u_k + \omega_k$,  with target set $K$ as defined in Section \ref{sec:origin_lqg}. In addition, in this example,  the avoid set includes  obstacles placed randomly within the state space as depicted in Figure \ref{fig:MIQPfig}. The safe set is $(K^\prime \setminus \bigcup_{j=1}^5 K^j_\alpha)$, where $K^\prime$ was defined in the previous example, and each $K^j_\alpha$ denotes a hyper-rectangular obstacle. We denote the union of obstacle sets by $K_\alpha=\bigcup_{j=1}^5 K^j_\alpha$. The reach-avoid time horizon is $T=7$. We use Algorithm \ref{algo:vfalgo} to approximate the optimal reach-avoid value function and compute the greedy policy. 

We chose the same basis function numbers, basis parameters, sampling and state-relevance measures as well as violation and confidence levels as in Section \ref{sec:origin_lqg}, shown in Table \ref{tab:compareMIQP}. We simulated the performance of the ADP controller starting from 100 different initial conditions, selected such that at least one obstacle blocks the direct path to the origin. For every initial condition we sampled 100 different noise trajectory realizations and applied the corresponding control policies computed through \eqref{eq:approx_pol}. We then computed the empirical ADP reach-avoid success probability (denoted by $V_\text{ADP}$) by counting the total number of trajectories that reach $K$ while avoiding reaching  the obstacles or leaving $K^\prime$.

Note that due to the presence of multiple obstacles, the LQG approach cannot be used as a heuristic for comparison. Nevertheless, the problem of reaching a target set without passing through any obstacles is an instance of a path planning problem and has been studied thoroughly for deterministic systems (see for example, \cite{borenstein1991vector,richards2002aircraft,van2006anytime}). For a benchmark comparison we use the approach of \cite{richards2002aircraft} and formulate the reach-avoid problem for the noise-free system as a constrained mixed logic dynamical system (MLD) \cite{bemporad1999control}. This problem can in turn be recast as a mixed integer quadratic program (MiQP) and solved to optimality using standard branch and bound techniques. To account for noise in the dynamics $\omega_k$, we used a heuristic approach as follows. We truncated the density function of the random variables $\omega_k$ at $95\%$ of their total mass and enlarged each obstacle set $K_\alpha$ by the maximum value of the truncated $\omega_k$ in each dimension. This resembles the robust (worst-case) approach to control design. 

Starting from the same initial conditions as in the ADP approach, we simulated the performance of the MiQP-based control policy on the 100 trajectory realizations used in the ADP controller. We implemented the policy in receding horizon by measuring the state at each horizon step. The empirical success probability of trajectories that reach $K$ while staying safe is denoted by $V_\text{MiQP}$.  The mean difference $\|V_\text{ADP}-{V}_\text{MiQP}\|$ is presented in Table \ref{tab:para_tableMPC} and is computed by averaging the corresponding empirical reach-avoid success probabilities over the initial conditions. As seen in this table, as the number of basis functions increases, $\|V_\text{ADP}-{V}_\text{MiQP}\|$ decreases. This can indicate that the reach-avoid value function approximation is increasing in accuracy. Note that for an increase in the planning horizon $T$, the number of binary variables (and hence the computational complexity) in MiQP grows exponentially, whereas in the LP-based reach-avoid approach, the computation effort grows linearly with the horizon. 
\begin{center}
\begin{table}
\centering
\begin{tabular}{lllll}
$M_k$&$\|{V}_\text{ADP}-V_{\text{MiQP}}\|$& Construction (sec)&LP solution (sec)& Memory (MB)\\
\hline
50 & 0.214& 1.67&0.18 & 0.784\\
100 &0.168 &5.59 &2.66&3.20\\
200 & 0.084& 22.0 &4.30&12.8\\
500 &  0.070 &130 &80.0 &80.0\\
1000 & 0.045&507 &1210 &320\\
\hline
\end{tabular}
\captionof{table}{Example 2 - Accuracy and computational requirements for $\dim\parens*{\XUs}=6$.}
\label{tab:para_tableMPC}
\end{table}
\end{center}

\vspace*{-.5cm}
\subsection{\textbf{Example 3}}
\label{sec:race_car}

\begin{figure}[b]
\begin{minipage}[t]{0.45\linewidth}
\centering
\vspace{0pt}
\includegraphics[width=7cm, height=5cm]{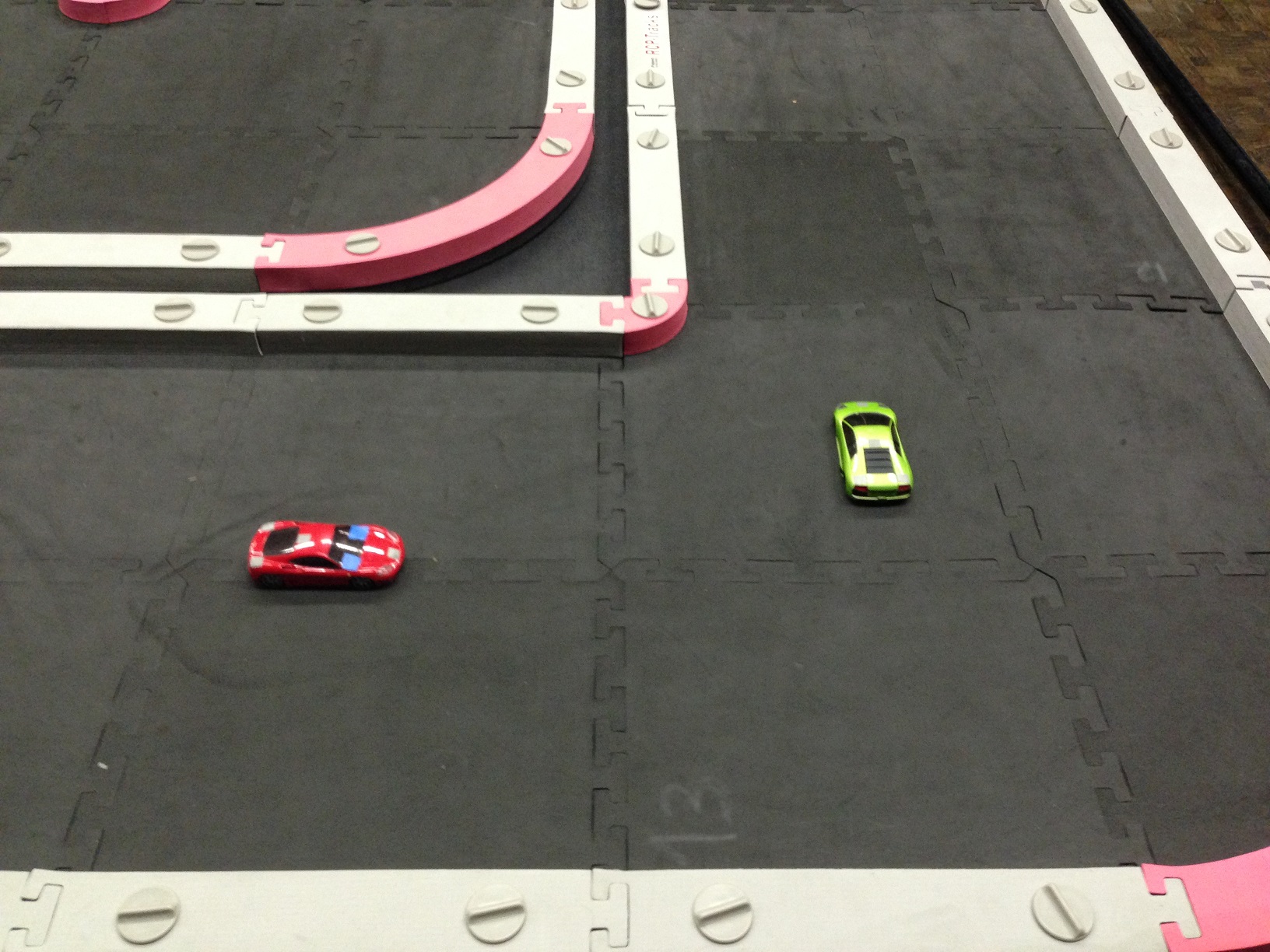}%
\caption{Example 3 - The set up of the Race-car cornering problem.}%
\label{fig:car_setup_pic}%
\end{minipage}\hfill
\begin{minipage}[t]{0.45\linewidth}
\centering
\vspace{0pt}
\includegraphics[width=7cm, height=5cm]{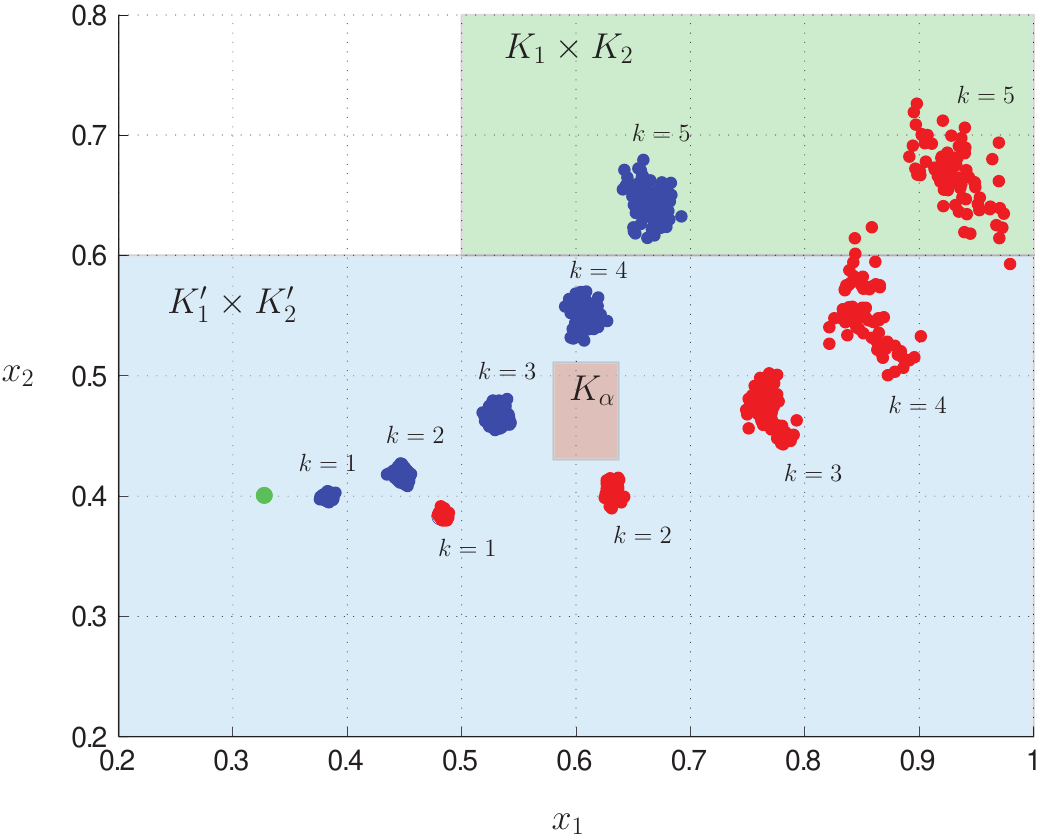}%
\caption{Example 3 - sample point clusters based on reach-avoid computation.}%
\label{fig:car_setup}%
\end{minipage}
\end{figure}
Consider the problem of driving a race car through a tight corner in the presence of static obstacles, illustrated in Figure \ref{fig:car_setup_pic}. As part of the ORCA project of the Automatic Control Lab (see \url{http://control.ee.ethz.ch/~racing/}), a six state variable nonlinear model with two control inputs has been identified to describe the movement of 1:43 scale race cars. The model derivation is discussed in \cite{liniger2014optimization} and is based on a unicycle approximation with parameters identified on the experimental platform of the ORCA project using model cars manufactured by Kyosho. We denote the state space by $\Xs\subset \mathbb{R}^6$, the control space by $\Us\subset\mathbb{R}^2$ and the identified dynamics by a function $f:\Xs\times\Us\mapsto \Xs$. The first two elements of each state $x\in\Xs$ correspond to spatial dimensions, the third to orientation, the fourth and fifth to body fixed longitudinal and lateral velocities and the sixth to angular velocity. The two control inputs $u\in\Us$ are the throttle duty cycle and the steering angle.

We will show how one can address the problem as a finite horizon reach-avoid problem and approximate its solution using the methodology presented. There are naturally several other approaches to address this problem  \cite{richards2002aircraft,couetoux2011continuous}. Our choice is only to illustrate the applicability of the framework for a general nonlinear dynamical system in high dimension. 

As  typically observed in practice, the state predicted by the identified dynamics and the state measurements recorded on the experimental platform are different due to process and measurement noise. Analyzing the deviation between predictions and measurements, we captured the uncertainties in the original model using additive Gaussian noise, $
g(x,u)=f(x,u)+\omega, \quad \omega\sim \mathcal{N}\parens{\mu,\Sigma}$, $\mu \in \mathbb{R}^6, \Sigma \in \m{S}_{++}^6$, where $\m{S}_{++}^6$ denotes the set of positive-definite matrices of dimension 6. The noise mean $\mu$, and diagonal covariance matrix $\Sigma$ have been selected such that the probability density function of the Markov decision process describing the uncertain dynamics resembles the empirical data obtained via measurements. As an example, Figure \ref{fig:noisefit} illustrates the fit for the angular velocity where $\mu_6=-0.26$ and $\Sigma(6,6)=0.53$. It follows that the kernel of the stochastic process is a GRBF with a single basis function described by the Gaussian distribution $\mathcal{N}\parens*{f(x,u)+\mu, \Sigma}$.

\begin{figure}[t]
\begin{minipage}[t]{0.45\linewidth}
\centering
\vspace{0pt}
\includegraphics[width=6cm, height=4cm]{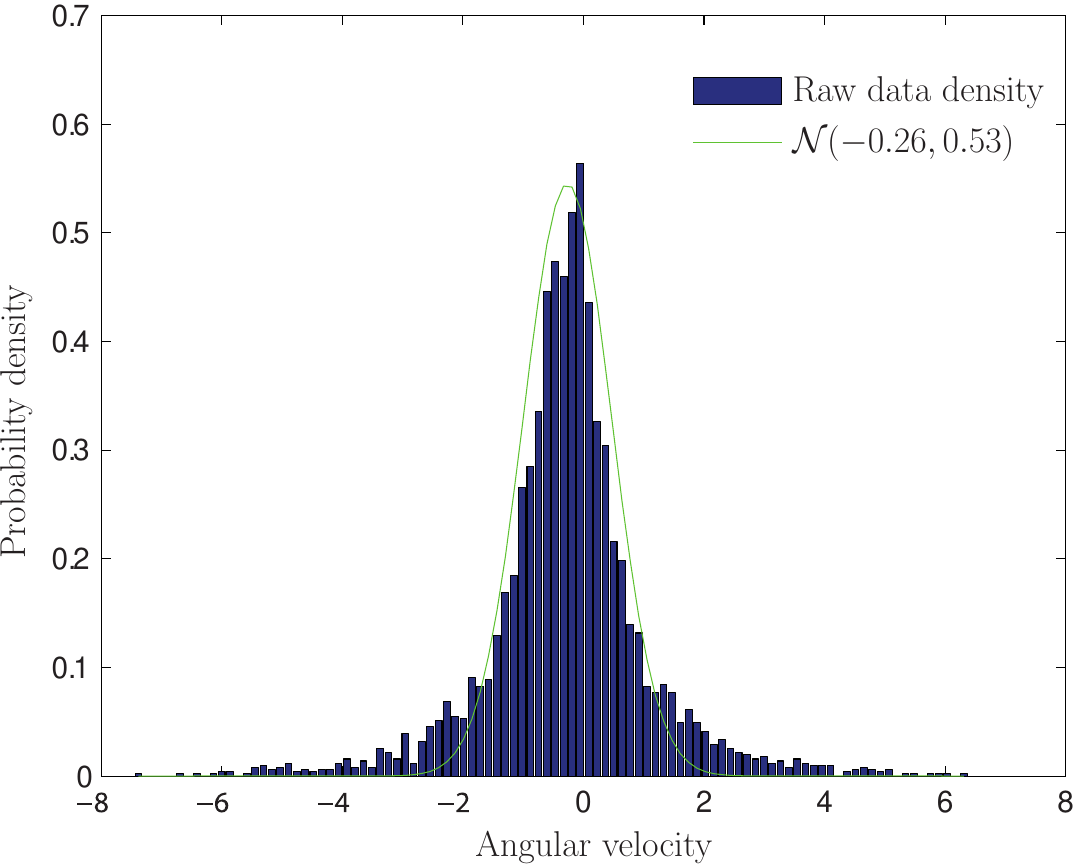}%
\caption{Empirical noise distribution.}%
\label{fig:noisefit}%
\end{minipage}\hfill
\begin{minipage}[t]{0.5\linewidth}
\centering
\vspace{0pt}
\begin{tabular}{l|rr|l}
{Safe region}&min&max&  variances\\
\hline
$K^\prime_1$ (m) & 0.2 & 1&[$8\times10^{-4}$,$1.2\times 10^{-3}$]\\
$K^\prime_2$ (m) & 0.2 & 0.6&[$8\times10^{-4}$,$1.2\times 10^{-3}$]\\
$K^\prime_3$ (rad)& $-\pi$ & $\pi$&[$5\times10^{-3}$,$1.5\times 10^{-2}$]\\
$K^\prime_4$ (m/s)& 0.3 & 3.5&[$5\times10^{-3}$,$1.5\times 10^{-2}$]\\
$K^\prime_5$ (m/s) & -1.5 & 1.5&[$5\times10^{-3}$,$1.5\times 10^{-2}$]\\
$K^\prime_6$ (rad/s) & -8 & 8&[2.00,4.00]\\
\hline
\end{tabular}
\captionof{table}{State constraints and basis functions' variances used in ADP approximation.}
\label{tab:state_limits}
\end{minipage}
\end{figure}
We cast the problem of driving the race car through a tight corner without reaching obstacles as a stochastic reach-avoid problem. Despite the highly nonlinear dynamics, the stochastic reach-avoid set-up can readily be applied to this problem. 

We consider a horizon of $T=6$ and a sampling time of $0.08$ seconds. The safe region of the spatial dimensions is defined as $\parens*{K^\prime_1\times K^\prime_2 }\setminus K_\alpha$ where $K_\alpha\subset \mathbb{R}^2$ denotes the obstacle, see Figures \ref{fig:car_setup_pic} and \ref{fig:car_setup}. The safe set in 6D is thus defined as $K^\prime=\parens*{\parens*{K^\prime_{1}\times K^\prime_{2}}\setminus K_\alpha} \times K^\prime_3\times K^\prime_4\times K^\prime_5\times K^\prime_6$ where $K^\prime_3,K^\prime_4,K^\prime_5,K^\prime_6$ describe the physical limitations of the model car (see Table \ref{tab:state_limits}). Similarly, the target set for the spatial dimensions is denoted by $K_1\times K_2$ and corresponds to the end of the turn as shown in Figure \ref{fig:car_setup}. The target set in 6D is then defined as  $K=K_1\times K_2\times K^\prime_3\times K^\prime_4\times K^\prime_5\times K^\prime_6$, which contains all states $x\in K^\prime$ for which $(x_1,x_2) \in K_1\times K_2$. The constraint sets are naturally decoupled over the state dimensions. Note that for practical purposes we have violated the assumption in Section \ref{sec:stochreach_DP} that the target set is a subset of the safe set in the spatial dimension (see Figure \ref{fig:car_setup}). The methodology and results remain the same if one extends the spatial safe set $K^\prime_1\times K^\prime_2$ to include $K_1\times K_2$.

We used  2000 GRBFs for each approximation step with centers and variances sampled according to uniform measures supported on $\bar{\Xs}$ and on the hyper-rectangle defined by the product of intervals in the rows of Table \ref{tab:state_limits}, respectively. We used a uniform state-relevance measure and a uniform sampling measure to construct each one of the finite linear programs in Algorithm \ref{algo:vfalgo}. All violation and confidence levels were chosen to be $\eps_k=0.2$ and $1-\beta_k=0.99$ respectively for $k=\{0,\dots,5\}$. We then implemented the steps of Algorithm \ref{algo:vfalgo} and compute a sequence of approximate value functions. 

To evaluate the quality of the approximations we initialized the car at two different initial conditions $
x^1=(0.33,0.4,-0.2,0.5,0,0)$ and $x^2=(0.33,0.4,-0.2,2,0,0)$. They 
correspond to entering the corner at low ($x^1_4=0.5 \text{ m/s}$) and high ($x^2_4=2\text{ m/s}$) longitudinal velocities. The approximate value functions evaluate to $\tilde{V}_0(x^1)=0.98$, $\tilde{V}_0(x^2)=1$ and indicate success with high probabilities for both cases. Interestingly,  the associated trajectories computed via the greedy policy defined through \eqref{eq:approx_pol} vary significantly. In the low velocity case, the car avoids the obstacle by driving above it while in the high velocity case, it does so by driving below it (see Figure \ref{fig:car_setup}). Such a behavior is expected since the car can slip if it turns aggressively at high velocities. We also computed empirical reach-avoid probabilities in simulation by sampling 100 noise trajectories from each initial state and implementing the ADP control policy using the associated value function approximation. The sample trajectories are plotted in Figure \ref{fig:car_setup} and the values were found to be $V_{\text{ADP}}(x^1)=1$ and $V_{\text{ADP}}(x^2)=0.99$

The controller was tested on the ORCA setup by running $10$ experiments from each initial condition. We pre-computed the control inputs at the predicted mean trajectory of the states over the horizon for each experiment. Implementing the feedback policy online would require solving problem \eqref{eq:approx_pol} within the sampling time of $0.08$ seconds. In theory, this computation is possible since the control space is only two dimensional but it requires developing an embedded nonlinear programming solver compatible with the ORCA setup. Here, we have implemented the open loop controller. We note however that if the open loop controller performs accurately, the closed loop computation can only improve the performance by utilizing updated state measurements. As demonstrated by the videos in \href{https://www.youtube.com/user/ETHZurichIfA}{(youtube:ETHZurichIfA)}, the car is successfully driving through the corner even when the control inputs are applied in an open loop. 


\section{Conclusions}
We developed a numerical approach to compute the value function of the stochastic reach-avoid problem for Markov decision processes with continuous state and action spaces. Since the method relies on solving linear programs we were able to tackle reach-avoid problems with larger dimensions than established state space gridding methods. The potential of the approach was analyzed through two benchmark case studies and a trajectory planning problem for a six dimensional nonlinear system with two  inputs. To the best of our knowledge, this is the first time that stochastic reach-avoid problems up to eight continuous state and input dimensions have been addressed. 

We are currently focusing on the problem of systematically choosing the  basis function parameters by exploiting knowledge about the system dynamics. Furthermore, we are developing decomposition methods for the large linear programs that arise in our approximation scheme to allow addressing control of MDPs with higher dimensions. We are also addressing tractable reformulations of the infinite constraints in the semi-infinite linear programs for stochastic reach-avoid problems to avoid sampling-based methods. Finally, given the close connections between reinforcement learning and approximate dynamic programming, our results may be useful in developing model-based reinforcement learning algorithms that incorporate safety considerations in addition optimizing a reward or cost function.

\section{Acknowledgements} 

The authors would like to thank Alexander Liniger from the Automatic Control Laboratory in ETH Z\"urich for his help in testing the algorithm on the ORCA platform

\small
\bibliography{ifaadp}
\bibliographystyle{apa}

\end{document}